\documentclass[12pt,a4paper,reqno]{article}
\usepackage[utf8]{inputenc}
\usepackage{tikz-network}
\usepackage{pgfplots}
\usepackage[english]{babel}
\usepackage{amsmath}
\usepackage{amssymb}
\usepackage{mathtools}
\usepackage{amsthm}
\usepackage{amsfonts}
\usepackage{graphicx}
\usepackage{graphics}
\usepackage{epsfig}
\usepackage{anysize}
\usepackage{lipsum}
\usepackage{wrapfig}
\usepackage{hyperref}
\usepackage{xcolor}
\usepackage{authblk}
\usepackage[export]{adjustbox}
\usepackage{geometry}
\usepackage{hyperref}
\newtheorem{thm}{Theorem}[section]
\newtheorem{prop}[thm]{Proposition}

\newtheorem{cor}[thm]{Corollary}
\newtheorem{defi}[thm]{Definition}

\newtheorem{rem}[thm]{Remark}

\usetikzlibrary{shapes,arrows,positioning}
\title{$k$-edge geodetic graphs}
\author[1]{Satyam Guragain\thanks{Email: shatym17@gmail.com.}}

\author[1]{Ravi Srivastava\thanks{Corresponding author, email: ravi@nitsikkim.ac.in.}}
\affil[1]{Department of Mathematics, National Institute of Technology Sikkim, South Sikkim 737139, India}
\date{}
\pgfplotsset{compat=1.18}
\begin{document}

\maketitle
\begin{abstract}
    A graph $G$ is $k$-edge geodetic graph if every edge of $G$ lies in at least one geodesic of length $k$. We studied some basic properties of $k$-edge geodetic graphs. We investigated the $k$ edge-geodeticity of complete bipartite graph $K_{m,n}$ and provide the minimum number of largest fixed order path that can cover $K_{m,n}$. We also studied the $k$-edge geodeticity of tree and the product graphs like Cartesian product, Strong product, Corona product, and provide the bounds for the minimum number of the largest fixed order path that can cover the graph.
\end{abstract}
\textbf{MSC Classification 2020:} 05C12, 68Q17
\section{Introduction}
 A graph $G$ is defined as an ordered pair $G=(V,E)$ where $V$ represents the set of vertices and $E$ represents the set of edges. In this work we only considered simple undirected graph, $i.e$, graphs without loops and multi-edges. The degree of a vertex $v$ is the number of edges incident to it and is denoted as $d(v)$. Eccentricity of a vertex $v$, denoted as $e(v)$, is the maximum distance of the vertex $v$ from the remaining vertices. A path in a graph is a finite of edges which joins a sequence of vertices which are all distinct. If the first and end vertices of a path is joined by an edge then the graph formed is a cycle. A tree is a connected graph without any cycle. A complete bipartite graph, denoted as $K_{n,m}$ is a type of bipartite graph where the vertex set can be divided into two disjoint sets $U$ and $V$ with $|U|=n$ and $|V|=m$ such that every edge in the graph connects a vertex in set $U$ to a vertex in set $V$ and every possible pair of vertices (one from set $U$ and one from set $V$) is connected by an edge.
 
Covering problems are among the fundamental problems in graph theory specially the vertex cover problem,
the edge cover problem, and the clique cover problem. An important subclass of covering problems is formed by the geodesic covering problem, i.e, coverings with shortest paths (also known as geodesics), e.g. the optimal transport flow in social networks requires an intensive study of geodesics \cite{appert2007measuring,jiang2004structural,thomson1995graph}. We introduce and study a related problem that we call $k$-edge geodetic problem.

Harary et al. \cite{harary1993geodetic} introduced the concept of the geodetic number of a graph, which serves as a measure of the connectivity and distribution of vertices within the graph. This concept plays a crucial role in understanding the minimum number of vertices required to ensure that all other vertices in the graph are part of some shortest path between them. Harary's work laid the foundation for subsequent research in this area, sparking interest in further exploring the properties and applications of geodetic numbers in different types of graphs. Santhakumaran and John \cite{santhakumaran2007edge} expanded on Harary's work by introducing the concept of the edge-geodetic number of a graph. The edge-geodetic number is particularly useful in analyzing network robustness and the efficiency of communication pathways in graphs. Manuel et al. further contributed to the field by defining the strong edge-geodetic cover \cite{manuel2017strong} and strong geodetic cover \cite{manuel2018strong} of a graph. These concepts emphasize the uniqueness of paths between vertices. Xavier et al. \cite{xavier2020some} explored the strong edge-geodetic number in the context of Cartesian products of complete graphs, providing bounds for this number. Gajavalli and Greeni \cite{gajavalli2024strong} continued the exploration of strong edge-geodetic numbers by examining lexicographic products of various graph types, including complete graphs ($K_m$), paths ($P_m$), cycles ($C_m$), and star graphs ($K_{1,m}$).
\section{Preliminary}
Consider a graph $G(V(G),E(G))$, with order $|V(G)|$ and size $|E(G)|$. A $(x -y)$ geodesic is the shortest path between the vertices $x$ and $y$. For a graph $G$, the length of the
maximum geodesic is called the graph diameter, denoted as $diam(G)$. Harary et al. introduced
a graph theoretical parameter in \cite{chartrand2002geodetic} called the geodetic number of a graph and it was further
studied in \cite{chartrand2002geodeticnew}. In \cite{chartrand2002geodetic} the geodetic number of a graph is defined as follows, let $I[u,v]$ be the set of all vertices lying on some $u-v$ geodesic of $G$ and $I[S]=\bigcup\limits_{u,v\in S}I[u,v]$. The set $S$ is called a geodetic set of $G$ if $I[S]=V(G)$. The cardinality of the minimum geodetic set of $G$ is the geodetic number $g(G)$ of $G$. Santhakumaran and John in \cite{santhakumaran2007edge} defined edge-geodetic set as, let $E[u,v]$ be the set of all edges lying on some $u-v$ geodesic of $G$, and for some non-empty set $S\subseteq V(G)$, $E[S]=\bigcup\limits_{u,v\in S}E[u,v]$. The set $S$ is called an edge-geodetic set of $G$ if $E[S]=E(G)$. The cardinality of the minimum edge-geodetic of $G$ is the edge-geodetic number $g_e(G)$ of $G$. Strong geodetic problem is a
variation of geodetic problem and is defined in \cite{manuel2018strong} as follows. For a graph $G(V(G),E(G))$, given
a set $S \subseteq V(G)$, for each pair of vertices $(x, y) \subseteq S$, $x \neq y$, let $\tilde{g}(x, y)$ be a selected fixed shortest
path between $x$ and $y$. Let $\tilde{I}(S) = \{\tilde{g}(x, y) : x, y \in S\}$ and $V(\tilde{I}(S)) =\bigcup\limits_{P\in \tilde{I}(S)}V(P)$. If $V(\tilde{I}(S)) = V(G)$
for some $\tilde{I}(S)$, then $S$ is called a strong geodetic set. The cardinality of the minimum strong
geodetic set is the strong geodetic number of $G$ and is denoted by $sg(G)$. The edge version of the
strong geodetic problem is defined in \cite{manuel2017strong} i.e. for a graph $G(V(G),E(G))$, a set $S \subseteq V(G)$ is called
a strong edge geodetic set if for any pair $x, y \in S$ a shortest path $P_{xy}$ can be assigned such that $\bigcup\limits_{x,y\in S}E(P_{xy}) = E(G)$. The cardinality of the smallest strong edge geodetic set of $G$ is called
the strong edge geodetic number and is denoted as $sg_e(G)$.

\section{Main results}
\begin{defi}
     A path $P$ in a graph $G$ is called $k$-geodesic if it is a shortest path of length $k$ between two vertices. 
 \end{defi}
 \begin{defi}
     A graph $G$ is $k$-edge geodetic graph if every edge of $G$ lies in at least one $k$-geodesic of $G$.
 \end{defi}
 \begin{rem}
     If a graph is $k$-edge geodetic than it is also $i$-edge geodetic for all $1\leq i \leq k$.
 \end{rem}
 \noindent We will use $[n]$-edge geodetic to represent a graph which is only $n$-edge geodetic but not $n+1$-edge geodetic.
If a graph $G$ is $[k]$-edge geodetic then the set of all $i$-geodesic of $G$ covers $G$ (for $i=1,2,\ldots,k$). The minimal set of $i$-geodesic which can cover $G$ is the $i$-geodesic generator of graph $G$ denoted as $gen_i(G)$. Also, we denote $gen_k(G)$ by $gen(G)$.
 \begin{prop}
     A cycle $C_{2n}$ is $[n]$-edge geodetic with $|gen(C_{2n})|=2$. 
 \end{prop}
 \begin{prop}
     A path $P_{n}$ is $[n]$-edge geodetic with $|gen(P_{n})|=1$. 
 \end{prop}
 \begin{prop}
     A cycle $C_{2n+1}$ is $[n]$-edge geodetic with $|gen(C_{2n+1})|=3$. 
 \end{prop}
 \begin{thm}
     A complete bipartite graph $K_{n,m}$ is $[2]$-edge geodetic unles $n,m=1$ with $|gen(K_{n,m})|= \left\lceil \frac{mn}{2} \right\rceil$. Here, $\left\lceil x \right\rceil$ represent least integer greater or equal to $x$.
 \end{thm}
 \begin{proof}
     Clearly, $K_{n,m}$ is $[2]$-edge geodetic unless $n,m=1$.\\
     To prove second part part we will consider three cases. Let $X$ and $Y$ denotes the partite set of vertices of $K_{n,m}$.\\
     Case I: Both $n=2a$ and $m=2b$ are even. $X=\{u_1,u_2,\ldots,u_{2a}\}$ and $Y=\{v_1,v_2,\ldots,v_{2b}\}$. Then, $$\bigcup_{j=1}^{b}\bigcup_{i=1}^{2a}v_{2j-1}u_{i}v_{2j}$$ covers all the edges of $K_{n,m}$. Also, for all $j=1,2,\ldots,b$ and $i=1,2,\ldots,2a$, $v_{2j-1}u_{i}v_{2j}$ have no common edge pairwise. So, $|gen(K_{n,m})|=2ab=\left\lceil \frac{mn}{2} \right\rceil$.\\
     Case II: One of $n,m$ is even and other is odd. $n=2a$, $m=2b+1$ and $X=\{u_1,u_2,\ldots,u_{2a}\}$, $Y=\{v_1,v_2,\ldots,v_{2b+1}\}$. Then, $$\bigcup_{j=1}^{b}\bigcup_{i=1}^{2a}v_{2j-1}u_{i}v_{2j}$$ covers all the edges of $K_{n,m}-v_{2b+1}$. By case I, $|gen(K_{n,m}-v_{2b+1})|=2ab$. Remaining edges $\{u_1v_{2b+1},u_2v_{2b+1},\ldots,u_{2a}v_{2b+1}\}$ in $K_{n,m}$ is covered by $a$ number of $2$-geodesics $u_1v_{2b+1}u_2, u_3v_{2b+1}u_4,\ldots,u_{2a-1}v_{2b+1}u_{2a}$. Thus,
     \begin{equation*}
         \left\{\bigcup_{j=1}^{b}\bigcup_{i=1}^{2a}v_{2j-1}u_{i}v_{2j}\right\}\bigcup\left\{\bigcup_{i=1}^{a}u_{2i-1}v_{2b+1}u_{2i}\right\}
     \end{equation*}
     covers all the edges of $K_{n,m}$ and also all these $2$-geodesics have no common edge pairwise. So, $|gen(K_{n,m})|=2ab+a=\left\lceil \frac{mn}{2} \right\rceil$.\\
     Case III: Both $n$ and $m$ are odd. $n=2a+1$, $m=2b+1$ and $X=\{u_1,u_2,\ldots,u_{2a+1}\}$, $Y=\{v_1,v_2,\ldots,v_{2b+1}\}$. By case II, $$\left\{\bigcup_{j=1}^{b}\bigcup_{i=1}^{2a}v_{2j-1}u_{i}v_{2j}\right\}\bigcup\left\{\bigcup_{i=1}^{a}u_{2i-1}v_{v_{2b+1}u_{2i}}\right\}$$ covers all the edges of $K_{n,m}-u_{2a+1}$. So, $|gen(K_{n,m}-u_{2a+1})|=a(2b+1)$. As, $|E(k_{n,m})|=(2a+1)(2b+1)$ (odd), we cannot cover all the edges of $K_{n,m}$ with $2$-geodesics that have no common edge pairwise, i.e, there exists at least one pair of $2$-geodesics which have a common  edge. Remaining edges $\{v_1u_{2a+1},v_2u_{2a+1},\ldots,v_{2b+1}u_{2a+1}\}$ in $K_{n,m}$ is covered by $b+1$ number of $2$-geodesics $v_1u_{2s+1}v_2, v_3u_{2s+1}v_4,\ldots,v_{2b-1}u_{2a+1}v_{2b}$. Thus,
     \begin{equation*}
         \left\{\bigcup_{j=1}^{b}\bigcup_{i=1}^{2a}v_{2j-1}u_{i}v_{2j}\right\}\bigcup\left\{\bigcup_{i=1}^{a}u_{2i-1}v_{2b+1}u_{2i}\right\}\bigcup\left\{\bigcup_{j=1}^{b}v_{2j-1}u_{2a+1}v_{2j}\right\}
     \end{equation*}
     covers all the edges of $K_{n,m}$ and also all these $2$-geodesics have no common edge pairwise except for pair $(v_{2b-1}u_{2a+1}v_{2b},v_{2b}u_{2a+1}v_{2b+1})$. So, $|gen(K_{n,m})|=a(2b+1)+b+1=\frac{(2a+1)(2b+1)+1}{2}=\left\lceil \frac{mn}{2} \right\rceil$.
 \end{proof}
 \begin{thm}\label{thmtree}
     A tree $T$ is $[k]$-edge geodetic where $k=\min\limits_{v \in S_T}e(v)$, $S_T$ is the set of pendant vertex of $T$.
 \end{thm}
 \begin{proof}
     For $k=\min\limits_{v\in S_T} e(v)$, there exists a pendant vertex $u$ in $T$ such that $\max\limits_{x\in V(T)}d(u,x)=k$. Thus the pendant edge from $u$ is not covered by any $k+1$-geodesic of $T$ and so $T$ is not $k+1$-edge geodetic. Clearly, all the pendant edges of $T$ are covered by some $k$-geodesic of $T$. Suppose $ab\in E(T)$ with $d_a,d_b>1$ and $d(u,a)=r,d(u,b)=r+1<k$. Also, we assume $ab$ is not covered by any $k$-geodesic with initial $u$. Take a $k$-geodesic $P:uu_1u_2\ldots u_k$ in $T$. Let $y\in S_T$ such that $q=\min\limits_{v\in S_T}d(v,b)=d(y,b)$, then there exists a $(u_{k+1-r-q} - y)$ $k$-geodesic which covers $ab$. 
 \end{proof}
 It is clear that in a $[k]$-edge geodetic tree $T$, each $k$-geodesics in $gen(T)$ covers at least one pendant edge of $T$. In other words, to cover $T$ it is sufficient to cover all the pendant edges of $T$. This leads us to conclude the following remark.
 \begin{rem}
     For a tree $T$, $|S_T|-1\geq |gen(T)|\geq \left\lceil \frac{|S_T|}{2}\right\rceil$.
 \end{rem}
 \begin{cor}
     For a tree $T$, $|gen(T)|=|S_T|-1$ if and only if either $|V(T)|\leq 3$ or there exists one and only one pendant vertex $u$ such that $e(u)=\min\limits_{v\in V(T)}e(v)$.
 \end{cor}
 \begin{thm}
      A connected graph $G$ with $|E(G)|\geq 2$ is $[k]$-edge geodetic ($k\geq 2$) if and only if $|V(G)|>|gen(G)|$.
 \end{thm}
 \begin{proof}
     Let $G$ is $[k]$-edge geodetic graph where $k\geq 2$ and $gen(G)=\{P_1,P_2,\ldots,P_r\}$. Further, suppose
     \[
     P_i: x_{i1}x_{i2}\ldots x_{ik} \text{ for all } 1\leq i \leq r. 
     \]
     Then for all $1\leq i\leq r$, $x_{i1}$ and $x_{ik}$ are two distinct vertices of $G$. Also, there exists at least one $x_{2i}$ in $P_2$ distinct from vertices of $P_1$ otherwise either $P_1\equiv P_2$ or $P_2$ is not a $k$-geodesic. We claim $P_2$ and $P_3$ contains at least two vertices distinct from vertices of $P_1$. If $P_2$ contains more than one vertex distinct from vertices of $P_1$ then the claim follows. If $P_2$ contains exactly one vertex say $x_{2t}$ distinct from vertices of $P_1$ then it follows that $P_3$ contains at least one vertex distinct from all the vertices of $P_1$ and $P_2$, otherwise either $P_3\equiv P_1$ or $P_3\equiv P_2$ or $$P_3:x_{11}x_{12}\ldots x_{1t-1}x_{1t}x_{2t}x_{1t+1}\ldots x_{1k}$$ which is not a $k$-geodesic, a contradiction. Thus, $|V(P_1)\cup V(P_2)\cup V(P_3)|>3$. Continuing this process we will obtain
     \[
     |V(G)|>|gen(G)|.
     \]
     Conversely, suppose $|V(G)|>|gen(G)|$. If $G$ is $[1]$-edge geodetic then $gen(G)=E(G)$. So, $|V(G)|>|E(G)|$, i.e, $G$ is a tree. By Theorem \ref{thmtree}, $1=\min\limits_{v\in S_G}e(v)$ which is not possible as $G$ is connected and $|E(G)|\geq 2$.
 \end{proof}
 \begin{cor}
     For a graph $G$, $\Delta(G)\leq 2|gen(G)|$ where $\Delta(G)=\max\limits_{v\in V(G)}d(v)$ is the maximum degree of a vertex $v$. 
 \end{cor}
  Let \( G = (V(G), E(G)) \) and \( H = (V(H), E(H)) \) be two graphs. The vertex set of the Cartesian product \( G \square H \) is $V(G \square H) = V(G) \times V(H)=\{(g,h);g\in V(G), h\in V(H)\}$ and 
the edge set \( E(G \square H) \) is given by
\[
E(G \square H) = \left\{(g_1, h_1)(g_2, h_2) \mid g_1 = g_2 \text{ and } h_1h_2 \in E(H) \text{ or }  h_1 = h_2 \text{ and } g_1g_2 \in E(G)\right\}.
\]
The vertex set of the strong product $G\boxtimes H$ of $G$ and $H$ is same as the vertex set of $G\square H$ but the edge set is given by
\begin{equation*}
    \begin{split}
        E(G \boxtimes H) = &\{(g_1, h_1)(g_2, h_2) \mid g_1 = g_2 \text{ and } h_1h_2 \in E(H) \text{ or }  h_1 = h_2 \text{ and } g_1g_2 \in E(G)\\ &~~~~~~~~~~~~~~~~~~~~~~~~~~~~~~~~~~~~~~~~ \text{ or } g_1g_2 \in E(G) \text{ and } h_1h_2 \in E(H)\}.
    \end{split}
\end{equation*}
 \begin{thm}\label{thmcar}
     Let $G$ be $[p]$-edge geodetic graph and $H$ be $[q]$-edge geodetic graph. Then $G\square H$ is $[p+q]$-edge geodetic graph.\\
     Also, the number of generator of $G\square H\leq |gen(G)|\cdot |V(H)|+|gen(H)|\cdot |V(G)|$.
 \end{thm}
 \begin{proof}
     For an edge of the form $(u_i,v_j)(u_i,v_k)$ in $G\square H$, there exists a $q$-geodesic $P_1:x_1x_2\ldots x_tx_{t+1}\ldots x_q$ in $H$ where $x_tx_{t+1}=v_jv_k$ which covers $v_jv_k$ and a $p$-geodesic $P_2:u_1u_2\ldots u_i\ldots u_p$ which covers vertex $u_i$. Then, $$(u_1,x_1)\ldots (u_1,x_t)(u_2,x_t)\ldots (u_i,x_t)(u_i,x_{t+1})\ldots (u_i,x_q)(u_{i+1},x_q)\ldots (u_p,x_q)$$ is a $p+q$-geodesic in $G\square H$ which covers $(u_i,v_j)(u_i,v_k)$. Similarly, for an edge $(u_l,v_j)(u_r,v_j)$ in $G\square H$ there exists a $p+q$-geodesic in $G\square H$ which covers $(u_l,v_j)(u_r,v_j)$. As $G$ is not $p+1$-edge geodetic, there exists an edge $uv$ in $G$ which is not covered by any $p+1$-geodesic of $G$. Suppose an edge $(u,x)(v,x)$ in $G\square H$ is covered by $p+q+1$-geodesic $P:(u_{i_1},v_{j_1})(u_{i_2},v_{j_2})\ldots(u_{i_t},v_{j_t})(u_{i_{t+1}},v_{j_{t+1}})\ldots(u_{i_{p+q+2}},v_{j_{p+q+2}})$ where $u_{i_t}=u, v_{j_t}=v_{j_{t+1}}$ and $u_{i_{t+1}}=v$. Clearly, $u_{i_1}, u_{i_2}, \ldots, u_{i_t}, u_{i_{t+1}},\ldots, u_{i_{p+r+2}}$ are not distinct. If there are $k(\geq p+2)$ distinct $u_i$ then there exists a $p+1$-geodesic which covers $uv$ in $G$, a contradiction. If there are $k(<p+2)$ distinct $u_i$ then $P$ cannot be of length $p+q+1$. Hence, $G\square H$ is $[p+q]$-edge geodetic.\\
     Also, $|gen(G)|\cdot |V(H)|$ number of $p+q$-geodesic covers all the edges of the form $(u_i,v_j)(u_l,v_j)$ and $|gen(H)|\cdot |V(G)|$ number of $p+q$-geodesic covers all the edges of the form $(u_i,v_j)(u_i,v_k)$. This completes the proof.
 \end{proof}
 \begin{thm}
     Let $G$ be $[p]$-edge geodetic graph and $H$ be $[q]$-edge geodetic graph. Then $G\boxtimes H$ is $[k]$-edge geodetic graph where $k=min\{p,q\}$.\\
     Further, if $p\leq q$ then, $|gen(G\boxtimes H)|\leq |V(H)|\cdot |gen(G)|+|V(G)|\cdot|E(H)|+2|gen(G)|\cdot|E(H)|$.
 \end{thm}
 \begin{proof}
     Suppose $p\leq q$. For an edge of the form $(u_i,v_j)(u_i,v_l)$ in $G\boxtimes H$ there exists a $p$-geodesic $x_1x_2\ldots x_p$ in $H$ that covers $v_jv_l$. Thus, $(u_i,x_1)(u_i,x_2)\ldots (u_i,x_p)$ is a $p$-geodesic in $G\boxtimes H$ which covers $(u_i,v_j)(u_i,v_l)$. Similarly, for an edge of the form $(u_i,v_j)(u_l,v_j)$ in $G\boxtimes H$ there exists a $p$-geodesic $y_1y_2\ldots y_p$ in $G$ that covers $u_iu_l$. Thus, $(y_1,v_j)(y_2,v_j)\ldots (y_p,v_j)$ is a $p$-geodesic in $G\boxtimes H$ which covers $(u_i,v_j)(u_l,v_j)$. Also, for an edge of the form $(u_i,v_j)(u_l,v_t)$ in $G\boxtimes H$ there exists $p$-geodesics $x_1x_2\ldots x_p$ in $G$ and $y_1y_2\ldots y_p$ in $H$ which covers $u_iu_l$ and $v_jv_t$ respectively. Thus, $(x_1,y_1)(x_2,y_2)\ldots (x_p,y_p)$ is a $p$-geodesic in $G\boxtimes H$ which covers $(u_i,v_j)(u_l,v_t)$. There exists an edge $uv$ in $G$ which is not covered by any $p+1$-geodesic in $G$. Following similar argument as in Theorem \ref{thmcar}, all the edges of the form $(u,x)(v,x)$ is not covered by any $p+1$-path in $G\boxtimes H$. Thus, $G\boxtimes H$ is $[p]$-edge geodetic graph. Similarly we can prove the result for $q<p$.
     Second part follows from the fact that $|gen(G)|\cdot |V(H)|$ number of $p$-geodesic covers all the edges of the form $(u_i,v_j)(u_l,v_j)$, $|E(H)|\cdot |V(G)|$ number of $p$-geodesic covers all the edges of the form $(u_i,v_j)(u_i,v_k)$ and $2|gen(G)|\cdot |E(H)|$ number of $p$-geodesic covers all the edges of the form $(u_i,v_j)(u_l,v_t)$. This completes the proof. 
 \end{proof}
 Let \( G = (V(G), E(G)) \) and \( H = (V(H), E(H)) \) be two graphs, where \( V(G) = \{u_1, u_2, \ldots, u_n\} \) and \( V(H) = \{v_1, v_2, \ldots, v_m\} \). The corona product \( G \circ H \) of $G$ and $H$ is obtained by taking one copy of $G$ and $n$ copies of $H$ and then joining all the vertices in the $i^{th}$ copies of $H$ to the $i^{th}$ vertex of $G$. We denote $i^{th}$ copy of $H$ by $H^i$ and $V(H^i)=\{v_1^i,v_2^i,\ldots,v_m^i\}$. The vertex set \( V(G \circ H) \) is given by
\[
V(G \circ H) = \{u_1,u_2,\ldots,u_n\} \bigcup \left\{\bigcup_{i=1}^{n} V(H^i)\right\}.
\]
The edge set \( E(G \circ H) \) is given by
\[
E(G \circ H) = E(G)\cup\left\{\bigcup\limits_{i=1}^{n} E(H^i)\right\} \cup \left\{\bigcup_{i=1}^{n}\bigcup\limits_{j=1}^{m} v_j^i u \right\}.
\]
 \begin{thm}
     Let $G$ be a graph without isolated vertex and $H$ be a graph with $|E(H)|\geq 2$ then $G\circ H$ is $[2]$-edge geodetic if and only if $H$ is $2$-edge geodetic.
 \end{thm}
 \begin{proof}
     Let $G\circ H$ is $[2]$-edge geodetic. If $H$ is not $2$-edge geodetic then there exists an edge $v_iv_k$ in $H$ which is not covered by any $2$-geodesic of $H$ and so edges $v_i^jv_k^j;j=1,2,\ldots,|v(G)|$ is not covered by any $2$-geodesic of $G\circ H$, a contradiction.\\
     Conversely, let $H$ is $2$-edge geodetic. Then,\\
     $(i)$ any arbitrary edge $u_ku_l$ in $G\circ H$ is covered by $2$-geodesic $u_ku_lv_1^l$ in $G\circ H$.\\
     $(ii)$ any arbitrary edge $v_i^ju_j$ in $G\circ H$ is covered by $2$-geodesic $v_i^ju_ju_l$ in $G\circ H$ for some $u_l\in N_G(u_j)$.\\
     $(iii)$ any arbitrary edge $v_i^jv_k^j$ in $G\circ H$ is covered by $2$-geodesic $v_i^jv_k^jv_l^j$ in $G\circ H$ where $v_iv_kv_l$ is a $2$-geodesic which covers $v_iv_k$ in $H$.\\
     Also, all the shortest paths between two vertices $v_i^j$ and $v_k^j$ in $G\circ H$ is less or equal to $2$. So, $G\circ H$ can not be covered by $3$-geodesics. 
 \end{proof}
 \begin{cor} For graphs $G$ and $H$ with $|V(G)|=n,|V(H)|=m$,\\
     (i) if $|E(G)|< nm$ then,
        \begin{equation*}
            \begin{split}
                \left\lceil |gen_2(H)|\cdot n+\frac{1}{2}|E(G)|+\frac{1}{2}nm\right\rceil \leq |gen(G\circ H)| \leq \left\lceil |gen_2(H)|\cdot n+|E(G)|+\frac{1}{2}nm\right\rceil
            \end{split}
        \end{equation*}
        (ii) If $|E(G)|\geq nm$ then,
        $\left\lceil |gen_2(H)|\cdot n+\frac{1}{2}|E(G)|\right\rceil\leq |gen(G\circ H)|\leq |gen_2(H)|\cdot n+|E(G)|$.
 \end{cor}
 \bibliographystyle{plain}
\bibliography{references.bib}

\begin{thebibliography}{10}

\bibitem{appert2007measuring}
Manuel Appert and Chapelon Laurent.
\newblock Measuring urban road network vulnerability using graph theory: the case of montpellier's road network.
\newblock {\em La mise en carte des risques naturels}, page 89p, 2007.

\bibitem{chartrand2002geodetic}
Gary Chartrand, Frank Harary, and Ping Zhang.
\newblock On the geodetic number of a graph.
\newblock {\em Networks: An International Journal}, 39(1):1--6, 2002.

\bibitem{chartrand2002geodeticnew}
Gary Chartrand, Edgar~M Palmer, and Ping Zhang.
\newblock The geodetic number of a graph: A survey.
\newblock {\em Congressus numerantium}, pages 37--58, 2002.

\bibitem{gajavalli2024strong}
S~Gajavalli and A~Berin Greeni.
\newblock On strong geodeticity in the lexicographic product of graphs.
\newblock {\em AIMS Mathematics}, 9(8):20367--20389, 2024.

\bibitem{harary1993geodetic}
Frank Harary, Emmanuel Loukakis, and Constantine Tsouros.
\newblock The geodetic number of a graph.
\newblock {\em Mathematical and Computer Modelling}, 17(11):89--95, 1993.

\bibitem{jiang2004structural}
Bin Jiang and Christophe Claramunt.
\newblock A structural approach to the model generalization of an urban street network.
\newblock {\em GeoInformatica}, 8:157--171, 2004.

\bibitem{manuel2017strong}
Paul Manuel, Sandi Klav{\v{z}}ar, Antony Xavier, Andrew Arokiaraj, and Elizabeth Thomas.
\newblock Strong edge geodetic problem in networks.
\newblock {\em Open Mathematics}, 15(1):1225--1235, 2017.

\bibitem{manuel2018strong}
Paul Manuel, Sandi Klav{\v{z}}ar, Antony Xavier, Andrew Arokiaraj, and Elizabeth Thomas.
\newblock Strong geodetic problem in networks.
\newblock {\em Discussiones Mathematicae Graph Theory}, 40(1):307--321, 2018.

\bibitem{santhakumaran2007edge}
AP~Santhakumaran and J~John.
\newblock Edge geodetic number of a graph.
\newblock {\em Journal of Discrete Mathematical Sciences and Cryptography}, 10(3):415--432, 2007.

\bibitem{thomson1995graph}
Robert~C Thomson and Dianne~E Richardson.
\newblock A graph theory approach to road network generalisation.
\newblock In {\em Proceeding of the 17th international cartographic conference}, pages 1871--1880, 1995.

\bibitem{xavier2020some}
D~Antony Xavier, Deepa Mathew, Santiagu Theresal, and Eddith~Sarah Varghese.
\newblock Some results on strong edge geodetic problem in graphs.
\newblock {\em Communications in Mathematics and Applications}, 11(3):403--413, 2020.

\end{thebibliography}
\end{document}